\documentclass{ifacconf}

\usepackage{graphicx}      % include this line if your document contains figures
\usepackage{natbib}        % required for bibliography
\usepackage[utf8]{inputenc}
\usepackage[T1]{fontenc}
\usepackage[english]{babel}
\usepackage{amsmath}
\usepackage{amsfonts}
\usepackage{amssymb}
\usepackage{makeidx}
\usepackage{lmodern}

\usepackage{color}
\usepackage{comment}
\usepackage[dvipsnames]{xcolor}
\usepackage{graphicx}
\usepackage{framed}
\usepackage{tikz}
\usepackage{calrsfs,fancybox}
\DeclareMathAlphabet{\pazocal}{OMS}{zplm}{m}{n}

\usepackage{fourier-orns}
\usepackage{mathtools}
\usepackage{relsize}
\usepackage{dsfont}

\newcommand{\R}{\mathbb{R}}

\newcommand{\Hpazo}{\pazocal{H}}

\newcommand{\Ppazo}{\pazocal{P}}

\newcommand{\Cpazo}{\pazocal{C}}

\newcommand{\Pcal}{\mathcal{P}}

\newcommand{\Var}{\textnormal{Var}}

\newcommand{\sign}{\textnormal{sign}}
\newcommand{\Lip}{\textnormal{Lip}}

\newcommand{\xb}{\boldsymbol{x}}

\newcommand{\pb}{\boldsymbol{p}}
\newcommand{\ub}{\boldsymbol{u}}
\newcommand{\vb}{\boldsymbol{v}}

\newcommand{\Pn}{(\Ppazo_N)}
\newcommand{\Pin}{(\Ppazo_\infty)}

\newcommand{\INTDom}[3]{\int_{#2} #1 \textnormal{d} #3}
\newcommand{\INTSeg}[4]{\int_{#3}^{#4} #1 \textnormal{d} #2}

\newcommand{\weak}{\rightharpoonup}
\newcommand{\eps}{\varepsilon}

\newenvironment{pschema}[1]{\vspace{0.2cm}
\noindent\begin{Sbox}\begin{minipage}{.95\columnwidth}\vspace{0.0cm}\begin{center}{ \large #1}\vspace{0.2cm}\\
\begin{minipage}{0.9\textwidth}}{\end{minipage}\end{center}\vspace{0.1cm}\end{minipage}\end{Sbox}\fbox{\TheSbox}\vspace{0.2cm}}
  
\newtheorem{rmk}{Remark}
\newtheorem{Def}{Definition}
\newenvironment{proof}{{\it Proof:}}{\hfill$\square$}

\begin{document}
\begin{frontmatter}

\title{Variance Optimization and Control Regularity for Mean-Field Dynamics\thanksref{footnoteinfo}} 

\thanks[footnoteinfo]{This research was supported by the Padua University grant SID 2018 ``Controllability, stabilizability and infimun gaps for control systems'', prot.BIRD 187147 and the STARS Grants program CONNECT: Control of Nonlocal Equations for Crowds and Traffic models}

\author[First]{Benoît Bonnet} 
\author[Second]{Francesco Rossi} 

\address[First]{Inria Paris, Laboratoire Jacques-Louis Lions, Sorbonne Université, Université Paris-Diderot SPC, CNRS, Inria, 75005 Paris, France\\ (e-mail benoit.a.bonnet@inria.fr)}
\address[Second]{Dipartimento di Matematica ``Tullio Levi-Civita'', Università degli Studi di Padova, 63 via Trieste, Padova, Italy\\ (e-mail: francesco.rossi@math.unipd.it)}

\begin{abstract}. We study a family of optimal control problems in which one aims at minimizing a cost that mixes a quadratic control penalization and the variance of the system, both for finitely many agents and for the mean-field dynamics as their number goes to infinity. While solutions of the discrete problem always exist in a unique and explicit form, the behavior of their macroscopic counterparts is very sensitive to the magnitude of the time horizon and penalization parameter.

\noindent When one minimizes the final variance, there always exists a Lipschitz-in-space optimal controls for the infinite dimensional problem, which can be obtained as a suitable extension of the optimal controls for the finite-dimensional problems. The same holds true for variance maximizations whenever the time horizon is sufficiently small. On the contrary, for large final times (or equivalently for small penalizations of the control cost), it can be proven that there does not exist Lipschitz-regular optimal controls for the macroscopic problem.

\end{abstract}

\begin{keyword} Mean-Field Optimal Control, Conservation Laws, Distributed Parameter Systems. %Large Scale Complex Systems
\end{keyword}

\end{frontmatter}

\section{Introduction}

The mathematical analysis of collective behavior in large-scale systems of interacting agents has received an increasing attention during the past decades. Multi-agent systems are ubiquitous in applications, ranging from networked control to animal flocks and swarms, see e.g. \cite{Bullo2009,CS1}.  In this context, a multi-agent system is usually described by a family of ordinary differential equations (ODEs for short), of the form
\begin{equation}
\label{eq:IntroODE}
\dot x_i(t) = \vb_N(t,\xb(t),x_i(t)), 
\end{equation}
where $\xb = (x_1,\dots,x_N)$ denotes the state of all the agents and $\vb_N : [0,T] \times (\R^d)^N \times \R^d \rightarrow \R^d$ is a \textit{non-local velocity field} depending both on the running agent and on the whole state of the system. However general and useful, the intrinsic dependence of such models on the number $N \geq 1$ of agents makes most of the classical computational approaches practically intractable for realistic scenarios. 

One of the most natural ideas to circumvent this limitation is to approximate the large system in \eqref{eq:IntroODE} by a single infinite-dimensional dynamics. This process, called the \textit{mean-field limit}, describes the evolution of the system when the number $N \geq 1$ of agents tends to infinity in a specific way (see e.g. the survey \cite{golse}). In this setting, agents are supposed to be identical or indistinguishable, and the assembly of particles is described by means of its \textit{spatial density} $\mu(\cdot)$, which is represented by a measure. The evolution in time of this global quantity is then prescribed by a \textit{non-local continuity equation}, that is a partial differential equation (PDE for short) of the form
\begin{equation}
\label{eq:IntroPDE}
\partial_t \mu(t) + \nabla \cdot ( v(t,\mu(t),\cdot)\mu(t)) = 0.
\end{equation}
This approach has been successfully used e.g. to model pedestrian dynamics and biological systems (\cite{Camazine2001,CPT}), as well as to transpose the study of classical patterns such as consensus or flocking to macroscopic approximations of discrete multi-agent systems (\cite{Carrillo2010,HaLiu}).

In addition to the modeling and analysis of this class of dynamics, the problem of controlling multi-agent systems is relevant in a growing number of applications, see e.g. \cite{Caponigro2015,leonard}. Motivated by implementability and efficiency considerations, many contributions have aimed at generalising relevant notions of control theory to PDEs of the form \eqref{eq:IntroPDE}, serving as mean-field approximations of the discrete systems \eqref{eq:IntroODE} (see e.g. \cite{bermansurvey}). A few articles have been dealing with controllability results (see \cite{Duprez2019,Duprez2020}) or explicit syntheses of control laws (e.g. \cite{Caponigro2015,ControlKCS}). On the other hand, the major part of the literature has been focusing on \textit{mean-field optimal control problems}, with contributions ranging from existence results (\cite{ContInc,FLOS,MFSOC,MFOC}) to first-order optimality conditions (\cite{MFPMP,PMPWassConst,SetValuedPMP,PMPWass} and references therein). 

In this article, we consider optimal control problems formulated both in the ODE and PDE settings, and discuss the possibility of applying the mean-field approach (i.e. to let $N \to +\infty$) on their solutions. We shall restrict our attention to a very simple family of problems, which exhibits the most important issues arising in this setting. Our goal in this context is to prove the following idea: if optimal controls at the discrete level can be written as Lipschitz functions of the individual agent states, with a Lipschitz constant that is uniform with respect to $N \geq 1$, then such controls pass to the limit, and the resulting mean-field optimal controls are Lipschitz as well. Instead, when the Lipschitz constants of the discrete optimal controls explode as $N\to+\infty$, then there does not exist a Lipschitz-regular minimizer for the mean-field problem.

To this end, we study one of the simplest optimal control problem possible for \eqref{eq:IntroODE} and \eqref{eq:IntroPDE}. We posit that the agents evolve on the real line $\R$, and that the controls act linearly on each of them. We further assume that there is no interaction between the agents at the dynamical level (i.e. $v\equiv 0$), and that a final cost promotes either the minimization or the maximization of the variance (both in the finite and infinite-dimensional setting) at time $T > 0$. Moreover, a running cost encodes an $L^2$-penalization of the controls. The relative weight between these two terms is represented by a scalar quantity $\lambda\neq0$, which value compared with respect to $T$ plays a fundamental role in the regularity of optimal controls, as amply discussed below.

In the sequel, we will therefore consider the following discrete multi-agent optimal control problem. 

\vspace{-0.2cm}

\begin{pschema}{$\Pn$}
Minimize the cost functional \smallskip
\begin{equation*}
\Cpazo_N(\xb^0,\ub):=\frac{1}{2N} \sum_{i=1}^N \INTSeg{\hspace{-0.15cm} u_i^2(t)}{t}{0}{T} - \frac{1}{2\lambda}\Var(\xb(T)), 
\end{equation*}
where \vspace{0.15cm}
\begin{itemize}
\item[$\diamond$] the controls $\ub:[0,T]\to [-1,1]^N$ are Lebesgue measurable, \smallskip
\item[$\diamond$] the curve $\xb(\cdot):=(x_1(\cdot),\dots,x_N(\cdot))$ is the unique solution of the controlled dynamics
\begin{equation}
\label{e-ODE}
\dot x_i(t) = u_i(t), \quad x_i(0)  =  x_i^0,
\end{equation}
with $x_i^0 := \tfrac{2i-N-1}{N-1}$ for each $i \in \{1,\dots,N\}$.
\end{itemize}
\vspace{-0.15cm}
\end{pschema}
The regularity required in $\Pn$ for the controls is the standard one ensuring existence and uniqueness of the solution to \eqref{e-ODE}, see e.g. \cite[Chapter 23]{Clarke}. %Indeed, the control does not depend on the state and it is measurable and bounded.

For the infinite-dimensional problem, defining admissible controls is more delicate, as one needs to ensure the well-posedness of the solution to \eqref{eq:IntroPDE}. For this reason, we impose a Lipschitz regularity of $(t,x) \mapsto u(t,x)$ with respect to the space variable $x \in \R$ (for a thorough discussion of this issue, see e.g. \cite{lipreg}), and study the following mean-field optimal control problem.

\vspace{-0.2cm}

\begin{pschema}{$\Pin$}
Minimize the cost functional \smallskip
\begin{equation*}
\hspace{-0.35cm} \Cpazo_\infty(\mu^0,u):= \frac{1}{2} \INTSeg{\hspace{-0.15cm} \INTDom{u^2(t,x)}{\R}{\mu(t)(x)}}{t}{0}{T} - \frac{1}{2\lambda}\Var(\mu(T)), 
\end{equation*}
\end{pschema}

\begin{pschema} \hfill \vspace{-0.15cm}
where \vspace{0.15cm}
\begin{itemize}
\item[$\diamond$] the controls $u : [0,T] \times \R \mapsto [-1,1]$ are measurable in time and Lipschitz in space,
 \item[$\diamond$] the curve $\mu(\cdot)$ is the unique solution of the controlled dynamics
\begin{equation}
\label{e-PDE}
\left\{
\begin{aligned}
& \partial_t \mu(t) + \nabla \cdot (u(t,\cdot) \mu(t)) = 0, \\
&\mu(0) = \mu^0,
\end{aligned}
\right.
\end{equation}
with $\mu^0 := \tfrac12 \chi_{[-1,1]}.$
\end{itemize}
\end{pschema}

In what follows, we will describe precisely in which sense the problem $\Pin$ is the limit of $\Pn$ as $N\to+\infty$. For the moment, observe that for each vector $\xb \in \R^N$ of $N \geq 1$ agent positions, one can define the {\it empirical measure}
\begin{equation*}
\mu_N:= \tfrac{1}{N} \mathsmaller{\sum}\limits_{i=1}^N \delta_{x_i} \quad \text{where} \quad \xb = (x_1,\dots,x_N).
\end{equation*}
Via this association, one can easily show that the discrete and continuous variances coincide, namely
\begin{equation*}
\begin{aligned}
\Var(\xb) :&= \tfrac{1}{N} \mathsmaller{\sum}\limits_{i=1}^N x_i^2 - \Big( \tfrac{1}{N} \mathsmaller{\sum}\limits_{i=1}^N x_i \Big)^2\\
& = \mathsmaller{\INTDom{x^2}{\R}{\mu_N(x)}} - \Big( \mathsmaller{\INTDom{x}{\R}{\mu_N(x)}} \Big)^2 = \Var(\mu_N),
\end{aligned}
\end{equation*}
and the initial data $(\mu_N^0)$ associated with $\xb^0$ as in $\Pn$ converge in the sense of measures \eqref{e-weak} towards $\tfrac{1}{2} \chi_{[-1,1]}$.

The terms involving the controls are more tricky, as some extra regularity is needed to ensure some sort of convergence between the discrete and continuous models. This is the crucial point of this article: we will show that the {\bf optimization process induces sufficient regularity of the optimal control for $\lambda \in (-\infty,0) \cup (T,+\infty)$}, while Lipschitz solutions for $\Pin$ fail to exist when $\lambda\in (0,T]$.

\begin{thm}[Main result]
\label{t-main} 
Let $\lambda>T$ or $\lambda<0$. Then, there exists a minimizer $u^* : [0,T] \times \R \rightarrow \R$ of $\Pin$ which is uniformly Lipschitz with respect to the space variable. Moreover, this minimizer is the limit of optimal controls $(u_i^*(\cdot))$ for $\Pn$, in the sense that $|x_i^*(t)- x| \to 0$ implies $|u_i^*(t) - u^*(t,x)| \to 0$. Instead for $\lambda\in(0,T]$, there does not exist a Lipschitz-in-space minimizer for $\Pin$.
\end{thm}

The dichotomy exposed above is paradigmatic of many mean-field optimal control problems, and is investigated in greater generality in \cite{lipreg}.

The structure of the article is the following. We first introduce the continuity equation and regularity issues for $\Pin$ in Section \ref{s-PDE}. We then explicitly solve $\Pn$ in Section \ref{s-Pn}, and proceed to rigorously study the limit $\Pn \to \Pin$ in Section \ref{s-proof}, proving Theorem \ref{t-main}. We finally draw some conclusions in Section \ref{s-conclusion}.

\section{Transport equations and mean-field optimal control}
\label{s-PDE}

In this section, we fix some notations and recall several results about Wasserstein distances, continuity equations and mean-field optimal control problems. 

We denote by $\Pcal_c(\R^d)$ the space of probability measures on $\R^d$ with compact support, endowed with the standard {\it weak topology of measures}, defined as 
\begin{equation}
\label{e-weak}
\mbox{$\mu_n \underset{n \rightarrow +\infty}{\weak} \mu \quad \text{if} \quad \INTDom{f(x)}{\R^d}{\mu_n(x)} \underset{n \rightarrow +\infty}{\longrightarrow} \INTDom{f(x)}{\R^d}{\mu(x)}, $}
\end{equation}
for every $f \in C^{\infty}_c(\R^d)$. We also denote by $\Lip(f)$ the Lipschitz constant of a Lipschitz continuous function, i.e.
\begin{equation*}
\Lip(f):=\sup_{x,y\in \mathrm{dom}(f), x\neq y}\frac{\|f(x)-f(y)\|}{\|x-y\|}.
\end{equation*}
We recall the definition of solution to continuity equations.
\begin{Def}
We say that $\mu(\cdot) \in C^0([0,T],\Pcal_c(\R^d))$ solves a \textit{continuity equation} with initial condition $\mu^0 \in \Pcal_c(\R^d)$ driven by a vector field $w: [0,T] \times \R^d\to \R^d$, i.e. 
\begin{equation}
\label{eq:TransportPDE}
\left\{
\begin{aligned}
& \partial_t \mu(t) + \nabla \cdot (w(t,\cdot)\mu(t)) = 0, \\
& \mu(0) = \mu^0,
\end{aligned}
\right.
\end{equation} 
if the following distributional identity holds
\begin{equation}
\label{eq:TransportPDE_Dist1}
\mbox{$\INTSeg{\INTDom{ \Big( \partial_t \xi(t,x) + \langle \nabla_x \xi(t,x) , w(t,x) \rangle \Big)}{\R^d}{\mu(t)(x)}}{t}{0}{T} = 0,$}
\end{equation}
for any $\xi \in C^{\infty}_c((0,T) \times \R^d)$. 
\end{Def}

The connection between continuity equations in infinite dimension and ODEs in finite dimension is colloquially known as the {\it method of characteristics}, and is supported by the two following statements.

\begin{Def} Let $f:\R^d\rightarrow\R^d$ be a Borel map. The {\it push-forward} $f_{\#}\mu$ of $\mu \in \Pcal(\R^d)$ is the measure satisfying 
\begin{equation*}
(f_{\#}\mu)(E):=\mu(f^{-1}(E)),
\end{equation*}
for every $E \subset \R^d$ such that $f^{-1}(E)$ is $\mu$-measurable.
\end{Def}

\begin{thm}[Method of characteristics]
Let $\mu^0\in \Pcal_c(\R^d)$ and  $w : [0,T] \times \R^d \rightarrow \R^d$ be a Carath\'eodory vector field that is locally Lipschitz and sublinear. Then, the continuity equation \eqref{eq:TransportPDE} admits a unique solution $\mu(\cdot)$, given by 
\begin{equation*}
\hspace{2.5cm} \mu(t)=(\Phi_t^{w})_{\#}\mu^0 \hspace{1.1cm} \text{for each $t \in [0,T]$}, 
\end{equation*} 
where $x \in \R^d \mapsto \Phi_t^{w}(x) \in \R^d$ is the \textit{flow map} of $w$.
\end{thm}

\begin{rmk}
It is known that weak solutions to continuity equations can exist in low regularity contexts, see \cite[Chapter 8]{AGS}. However, the corresponding notions do not ensure the well-posedness of \eqref{eq:IntroPDE} for arbitrary measures, and are less suited to mean-field control.
\end{rmk}

\subsection{Wasserstein distance}

We now recall the definition of the Wasserstein distances (see e.g. \cite[Chapter 7]{AGS}), together with some of their connections to solutions of continuity equations. We will only work with the 1-Wasserstein distance, as this is sufficient for our subsequent developments.

\begin{Def}
Given $\mu,\nu\in \Pcal_c(\R^d)$, the $1$-Wasserstein distance between $\mu$ and $\nu$ is defined by
\begin{equation}
\label{e-W}
W_1(\mu,\nu):=\sup\limits_{\Lip(f)\leq 1} \mathsmaller{\INTDom{f(x)}{\R^d}{(\mu-\nu)(x)}}.
\end{equation}
\end{Def}
A first fundamental property of the $1$-Wasserstein distance is that it metrizes the weak convergence of measures \eqref{e-weak}, in the following sense.
\begin{prop}
It holds $\lim\limits_{n\to+\infty}W_1(\mu_n,\mu)=0$ if and only if
\begin{equation*}
\mbox{$\mu_n \underset{n \rightarrow +\infty}{\weak} \mu \quad \text{and} \quad \INTDom{|x|}{\R^d}{\mu_n(x)} \underset{n \rightarrow +\infty}{\longrightarrow} \INTDom{|x|}{\R^d}{\mu(x)}.$}
\end{equation*}
\end{prop}

We now recall a useful stability result with respect to the Wasserstein distance for solutions of \eqref{eq:TransportPDE}.

\begin{prop}\label{p-gronwall}
Let $\mu,~\nu\in \Pcal_c(\R^d)$ and $w: [0,T] \times \R^d \rightarrow\R^d$ be a uniformly bounded, measurable in time and Lipschitz in space vector field, with Lipschitz constant equal to $L \geq 0$. Then for each $t \in [0,T]$, it holds
\begin{equation}
\label{e-Gronwall}
W_1((\Phi_t^w)_{\#}\mu,(\Phi_t^w)_{\#}\nu) \leq e^{L t} W_1(\mu,\nu).
\end{equation}
\end{prop}

\begin{proof}
See e.g. \cite{Pedestrian} or \cite{ContInc} for a more general statement. 
\end{proof}

\section{Solutions of $\Pn$} \label{s-Pn}

In this section, we explicitly compute the solutions of $\Pn$. We use bold notations $(\xb,\pb) \in \R^{2N}$ to denote vectors in $\R^N$, as well as $\ub \in [-1,1]^N$. We choose $N \geq 2$ to be even, which ensures that $x_i(0)\neq 0$. This condition is not crucial for our result, but simplifies the discussion.

Observe first that $\Pn$ is regular both with respect to the dynamics and control variables. Moreover, optimal controls do exist since the set of admissible controls $ [-1,1]^N$ is convex and compact (see e.g. \cite[Theorem 23.11]{Clarke}). Finally, the smoothness of the data allows us to compute optimal controls via the Pontryagin Maximum Principle, see e.g. \cite[Chapter 22]{Clarke}. 

Denoting by $\pb$ a costate variable associated to a state vector $\xb$, the Hamiltonian of problem $\Pn$ writes as
\begin{equation}
\label{eq:ExplicitExp_Hamiltonian}
\Hpazo_N(\xb,\pb,\ub) := \mathsmaller{\sum}\limits_{i=1}^N \big( p_i u_i - \tfrac{1}{2N} u_i^2 \big).
\end{equation}
Given an optimal trajectory-control pair $(\xb^*_N(\cdot),\ub^*_N(\cdot))$, the PMP provides the existence of a curve $\pb^*_N(\cdot)$ satisfying
\begin{equation}
\label{e-PMP}
\begin{cases}
\dot p_i^*(t) & = 0, \\
p_i^*(T) & = \tfrac{1}{\lambda}\partial_{x_i} \Var(\xb^*_N(T)) = \tfrac{1}{N \lambda}(x_i^*(T) - \bar{\xb}_N^*(T)),  \\ 
u_i^*(t) & \in \underset{v \in [-1,1]}{\textnormal{argmax}} \, [ \, p_i^*(t) v - \tfrac{1}{2N}v^2 \, ],
\end{cases}
\end{equation}
where $\bar{\xb}=\frac1N\sum_{i=1}^Nx_i$ is the mean value of the vector $\xb$. Thus, the adjoint vector $\pb_N^*(\cdot)$ is constant, and satisfies
\begin{equation*}
p_i^*(t) = \tfrac{1}{\lambda N}(x_i^*(T) - \bar{\xb}_N^*(T)),
\end{equation*}
for any $i \in \{1,\dots,N\}$ and all times $t \in [0,T]$. As a consequence of the maximization condition, one can express the components of the optimal control $\ub^*(\cdot)$ explicitly as
\begin{equation}
\label{e-ui1}
u_i^*(t) = \pi \Big( \tfrac{1}{\lambda} (x_i^*(T) - \bar{\xb}_N^*(T)) \Big). 
\end{equation}
Here, we denoted by $u \in \R \mapsto \pi(u) \in [-1,1]$ the projection onto the set of admissible controls, namely 
\begin{equation*}
\pi(u):=\begin{cases}
1 &\mbox{~~if~~} u \geq 1, \\
u &\mbox{~~if~~} u \in [-1,1], \\
-1 &\mbox{~~if~~} u \leq -1.
\end{cases}
\end{equation*}
Our goal now is to analytically solve $\Pn$. First remark that the optimal control $\ub^*_N(\cdot) \equiv \ub^*_N$ does not depend on time, and let us denote by $\bar\ub^*_N$ its mean value. The mean value of the state then satisfies $\dot{\bar{\xb}}^*_N(t)=\bar\ub^*_N$, which implies
\begin{equation*}
\bar{\xb}_N^*(T) = T \bar{\ub}_N^* \qquad \text{since} \qquad \bar{\xb}_N(0) =0. 
\end{equation*}
Then, one can show that $\bar{\ub}^*_N = 0$, e.g. by observing that choosing $\tilde{u}_i := (u_i^* - \bar{\ub}^*_N)$ instead of $u_i^*$ would have no impact on the final cost while reducing the control cost. Consequently $\bar{\xb}^*(t)=\bar{\xb}(0)=0$ for all  $t \in [0,T]$, and since $x_i^*(T)=x_i(0)+Tu_i^*$ equation \eqref{e-ui1} now reads as
\begin{equation}
\label{e-ui}
u_i^* = \pi \Big( \tfrac{1}{\lambda}(x_i(0)+Tu_i^*) \Big). 
\end{equation} 
Now, we need to study the problem in several situations depending on the value of the penalization parameter $\lambda$.

\subsection{The case $\lambda>T$}

Consider the right-hand side of \eqref{e-ui} as a function of $u_i$, and observe that it is increasing, with maximal slope $T/\lambda<1$. Since the left-hand side is the identity, then \eqref{e-ui} always admits a unique solution (see Figure \ref{fig1} below).

\begin{figure}[htb]
\begin{center}
\includegraphics[width=8.4cm]{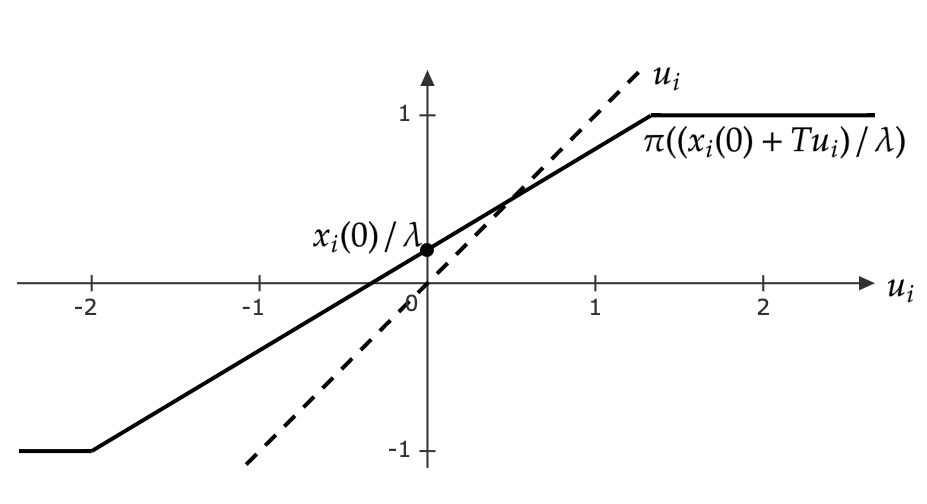}    % The printed column width is 8.4 cm.
\caption{{ \small Solution of \eqref{e-ui} for $\lambda>T$.}} 
\label{fig1}
\end{center}
\end{figure}

If $|x_i(0)+Tu_i^*| \leq \lambda$, then the projection coincides with the identity operator. In this case, equation \eqref{e-ui} reads as $\lambda u_i^* = x_i(0) + T u_i^*$. We can then write $u_i^* = \frac{x_i(0)}{\lambda-T}$ whenever such formula yields a control $u_i \in [-1,1]$, i.e. for $x_i(0)\in [-(\lambda-T),\lambda-T]$. For other indices, one can easily check that the unique solution for $x_i(0)>\lambda-T$ is $u_i^* = 1$. Similarly for $x_i(0)<-(\lambda-T)$, the unique solution to \eqref{e-ui} is  $u_i^* = -1$. This gives the following optimal controls
\begin{equation}\label{e-c1}
u_i^* = \begin{cases}
1& \mbox{~~if~~} x_i(0)>\lambda-T, \\
\frac{x_i(0)}{\lambda-T} & \mbox{~~if~~} x_i(0)\in [-(\lambda-T),\lambda-T], \\
-1& \mbox{~~if~~} x_i(0)<-(\lambda-T).
\end{cases}
\end{equation}

\subsection{The case $\lambda=T$}

In this case, equation \eqref{e-ui} has no solution whenever the projection coincides with the identity, as $x_i(0)\neq 0$ for every $i \in \{1,\dots,N\}$. A simple computation shows that the only solution is the following
\begin{equation}\label{e-c2}
u_i^* = \begin{cases}
1& \mbox{~~if~~} x_i(0)>0, \\
-1& \mbox{~~if~~} x_i(0)<0. 
\end{cases}
\end{equation}

\subsection{The case $\lambda\in(0,T)$}

In this case, equation \eqref{e-ui} does not have a uniquely defined solution when the projection operator coincides with the identity. Indeed for $x_i(0) \in [-(T-\lambda),T-\lambda]$, both $\frac{x_i(0)}{\lambda-T}$ and $\sign(x_i(0))$ are solutions. One can prove that this second choice provides the minimizer of the cost, and that the corresponding optimal control is given by \eqref{e-c2}.

\subsection{The case $\lambda<0$}

In this scenario, one aims at minimizing both $\Var(\xb(T))$ and the running control cost. One can show that the Pontryagin Maximum Principle still reads as \eqref{e-PMP}, and thus \eqref{e-ui} holds too. Therein, the right-hand side is decreasing and bounded, and \eqref{e-ui} always admits a unique solution. Direct computations, similar to the previous ones, allow to prove that the optimal controls are given by 
\begin{equation}
\label{e-c4}
u_i^* = \begin{cases}
-1& \mbox{~~if~~} x_i(0)>-(\lambda-T),\\
\frac{x_i(0)}{\lambda-T} & \mbox{~~if~~} x_i(0)\in [(\lambda-T),-(\lambda-T)],\\
1& \mbox{~~if~~} x_i(0)<\lambda-T.
\end{cases}
\end{equation}

\subsection{Comparison}

We now highlight two important features of the optimal controls written above. First, for each pair of parameters $(\lambda,T)$, the value of $u_i^*$ only depends on $x_i(0)$ and not on the actual number $N \geq 1$ of agents. This will play a crucial role in the following Section \ref{s-proof}.

Second, we aim to evaluate the following quantity, that can be seen as the Lipschitz constant of the optimal control
\begin{equation*}
L(t):=\max_{i\neq j}\frac{|u_i^* - u_j^*|}{|x_i^*(t)-x_j^*(t)|}. 
\end{equation*}
By recalling that $u_i^*,u_j^*$ are constant in time and that $x_i^*(t) = x_i(0)+t u_i^*$, we can isolate the following scenarios.

\begin{itemize}
\item $\lambda>T$ : Here $|x_i^*(t)-x_j^*(t)|$ is increasing in time, thus
\begin{equation*}
L(t)\leq L(0)\leq \frac{x_i(0)-x_j(0)}{(\lambda-T)(x_i(0)-x_j(0))}=\frac{1}{\lambda-T}.
\end{equation*}
Remark that in this case, the Lipschitz constant is uniformly bounded with respect to $N \geq 1$.
\item $\lambda\in (0,T]$ : In this case, one can easily see that the maximum is reached at $i=\frac{N}2$ and $j=i+1$, i.e  for the maximal negative initial position and the minimal positive one. It then holds 
\begin{equation*}
\begin{aligned}
L(t) & = \frac{|1-(-1)|}{|(x_i(0)-t)-(x_j(0)+t)|} =\frac{N-1}{1+t(N-1)}.
\end{aligned}
\end{equation*}
Contrary to the previous case, this constant depends explicitly on $N$ and can be arbitrarily large as $N\to +\infty$ when $t \in [0,T]$ is small.
\item $\lambda<0$ : The result is similar to the case $\lambda>T$. Here, the maximal value is attained for $t=T$, hence
\begin{equation*}
\begin{aligned}
& L(t) \leq L(T) =\\
&\frac{| x_i(0)-x_j(0)|}{\big|(\lambda-T)\big( x_i(0)-x_j(0) + T \big(\frac{x_i(0)-x_j(0)}{\lambda-T} \big) \big) \big|} =\frac{1}{|\lambda|}.
\end{aligned}
\end{equation*}
Also in this case, the Lipschitz constant is independent of $N \geq 1$ and uniformly bounded.
\end{itemize}

\section{Solution of $\Pin$}
\label{s-proof}

In this section, we prove our main result Theorem \ref{t-main}. The interest of the proof is two-fold. First, it shows the fundamental role played by the parameters $(\lambda,T)$ in the existence of regular solutions to $\Pin$. Second, it leverages quite simply and directly the explicit form of the solution to the discretized problems $\Pn$ derived above.

The idea of the proof is the following. When $\lambda>T$ or $\lambda<0$, we have seen that there exists a Lipschitz minimizer for $\Pn$, and we can use it to build a regular minimizer for $\Pin$. On the contrary when $\lambda\in(0,T]$, the minimizers of $\Pn$ are not regular and we contradict the optimality of any candidate Lipschitz minimizer for $\Pin$ by studying the mean-field limit of the discretized problems.

With this goal in mind, we explicitly build controlled vector fields that will be optimal for $\Pn$ and possibly for $\Pin$. We will define a candidate optimal vector field $u^* : [0,T] \times \R^d \rightarrow \R^d$ by requiring that
\begin{equation} 
\label{e-feedback}
u^*(t,x_i^*(t))=u_i^*(t),
\end{equation}
for every $i \in \{1,\dots,N \}$ and almost every $t \in [0,T]$. 

Therefore, we need to consider the following three cases.
\begin{itemize}
\item[$\bullet$] $\lambda>T$ : The optimal controls $u_i^*(\cdot)$ satisfy \eqref{e-c1}, and we first consider the case $x_i(0)>\lambda -T$. By setting 
\begin{equation*}
u^*(t,y) = 1, ~~ y\geq \lambda-T+t, 
\end{equation*}
one then has $u^*(t,x_i(0)+t)=1$, so that \eqref{e-feedback} holds. The case $x_i(0)<-(\lambda-T)$ is completely similar. If $x_i(0)\in [-(\lambda-T),\lambda-T]$, then \eqref{e-feedback} reads as 
\begin{equation*} 
u^* \left(t , x_i(0)+t\frac{x_i(0)}{\lambda-T} \right) = \frac{x_i(0)}{\lambda-T}.
\end{equation*}
This condition is verified by choosing the velocity field
\begin{equation*}
u^*(t,y) := \frac{y}{\lambda-T+t}, ~~ \text{ $y \in[-(\lambda-T+t),\lambda-T+t]$}.
\end{equation*}
These three cases can be merged into the formula
\begin{equation}
\label{e-f1}
u^*(t,y) = \pi\left(\frac{y}{\lambda-T+t}\right).
\end{equation}
\item[$\bullet$]$\lambda\in (0,T]$ : We now consider from \eqref{e-c2}, and set 
\begin{equation}
\label{e-f2}
u^*(t,y)=\sign(y), ~~ y\in(-\infty,-t)\cup(t,+\infty).
\end{equation}
We do not define $u^*(t,y)$ for $y\in[-t,t]$. 
\item[$\bullet$] $\lambda<0$ : Starting from \eqref{e-c4}, we easily show that
\begin{equation*}
u^*(t,y) = \pi\left(\frac{y}{\lambda-T+t}\right), 
\end{equation*}
satisfies \eqref{e-feedback}.
\end{itemize}

\begin{figure}[htb]
\hspace{-0.7cm}
\begin{tikzpicture}
%%%%%%%%%%%%%%%%%%%%%%%%%%%%%%%%%%%%%%%%%%%
% 		 FIRST PICTURE AND LEGEND         %
%%%%%%%%%%%%%%%%%%%%%%%%%%%%%%%%%%%%%%%%%%%
% Incorporating the picture 
\node(image) at (0,0)
     {\includegraphics[width = 0.56\textwidth]{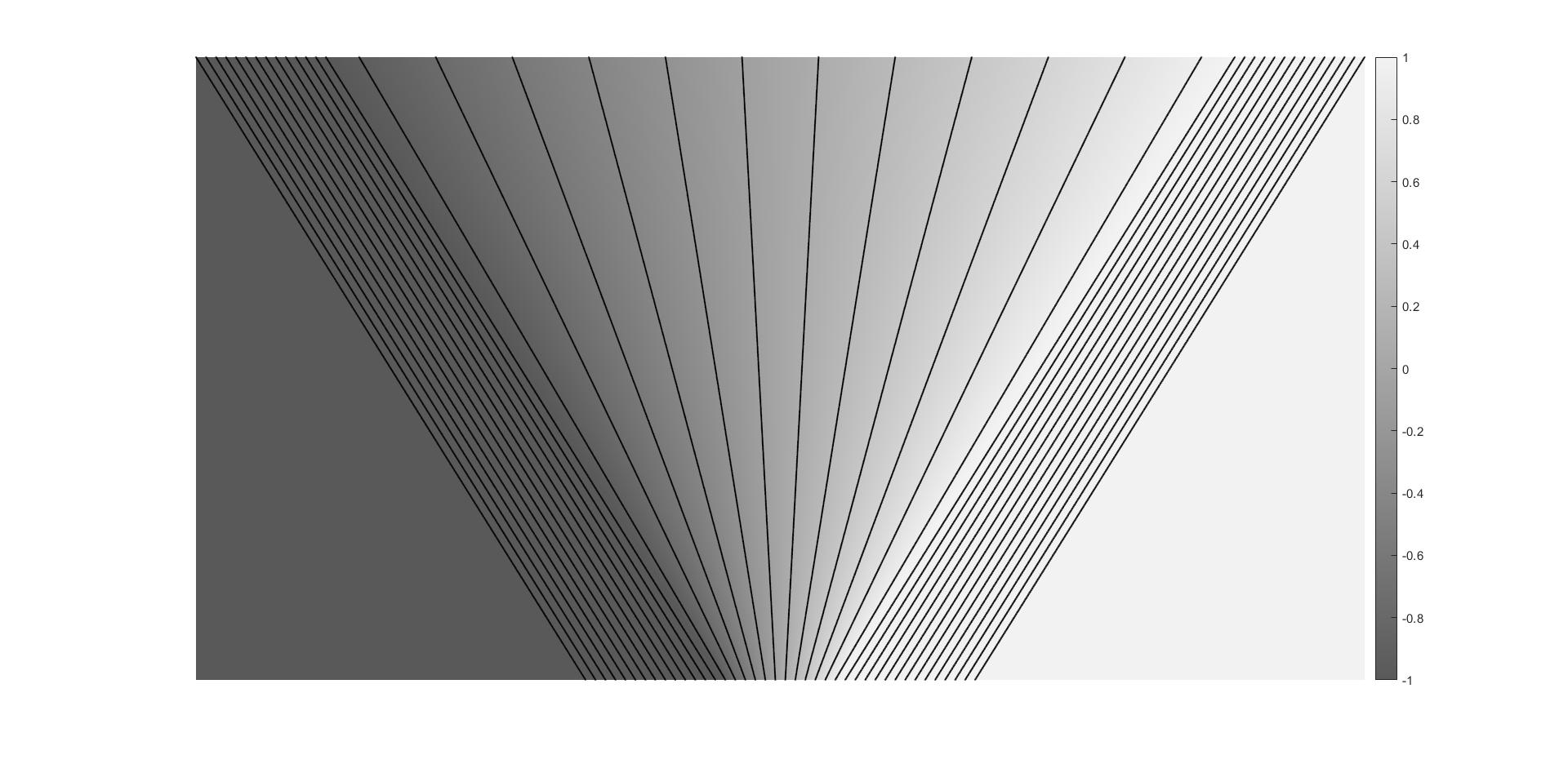}};
% Adding some axes
\draw[black,->](-4,-1.95) -- (4.095,-1.95);
\draw[black,->](-3.825,-2.2) -- (-3.825,2.25); 
% x-axis legend
\draw (-3.275,-2.35) node {\scriptsize $y=-(1+T)$};
\draw (3.5,-2.35) node {\scriptsize $y=1+T$};
% y-axis legend
\draw (-4.15,-1.85) node {\scriptsize \rotatebox{90}{$t = 0$}};
\draw (-4.15,2.15) node {\scriptsize \rotatebox{90}{$t = T$}};
% Labelling some agents
\draw (-1.35,-2.2) node {\scriptsize $-1 = x_1^0$};
\draw (-0.05,-2.2) node {\scriptsize $\ldots  \ldots \ldots$};
\draw (1.2,-2.2) node {\scriptsize $x_N^0 = 1$};
%
%%%%%%%%%%%%%%%%%%%%%%%%%%%%%%%%%%%%%%%%%%%
% 		 SECOND PICTURE AND LEGEND        %
%%%%%%%%%%%%%%%%%%%%%%%%%%%%%%%%%%%%%%%%%%%
\begin{scope}[shift = {(0,-5.5)}]
% Incorporating the picture
\node(image) at (0,0)
     {\includegraphics[width = 0.56\textwidth]{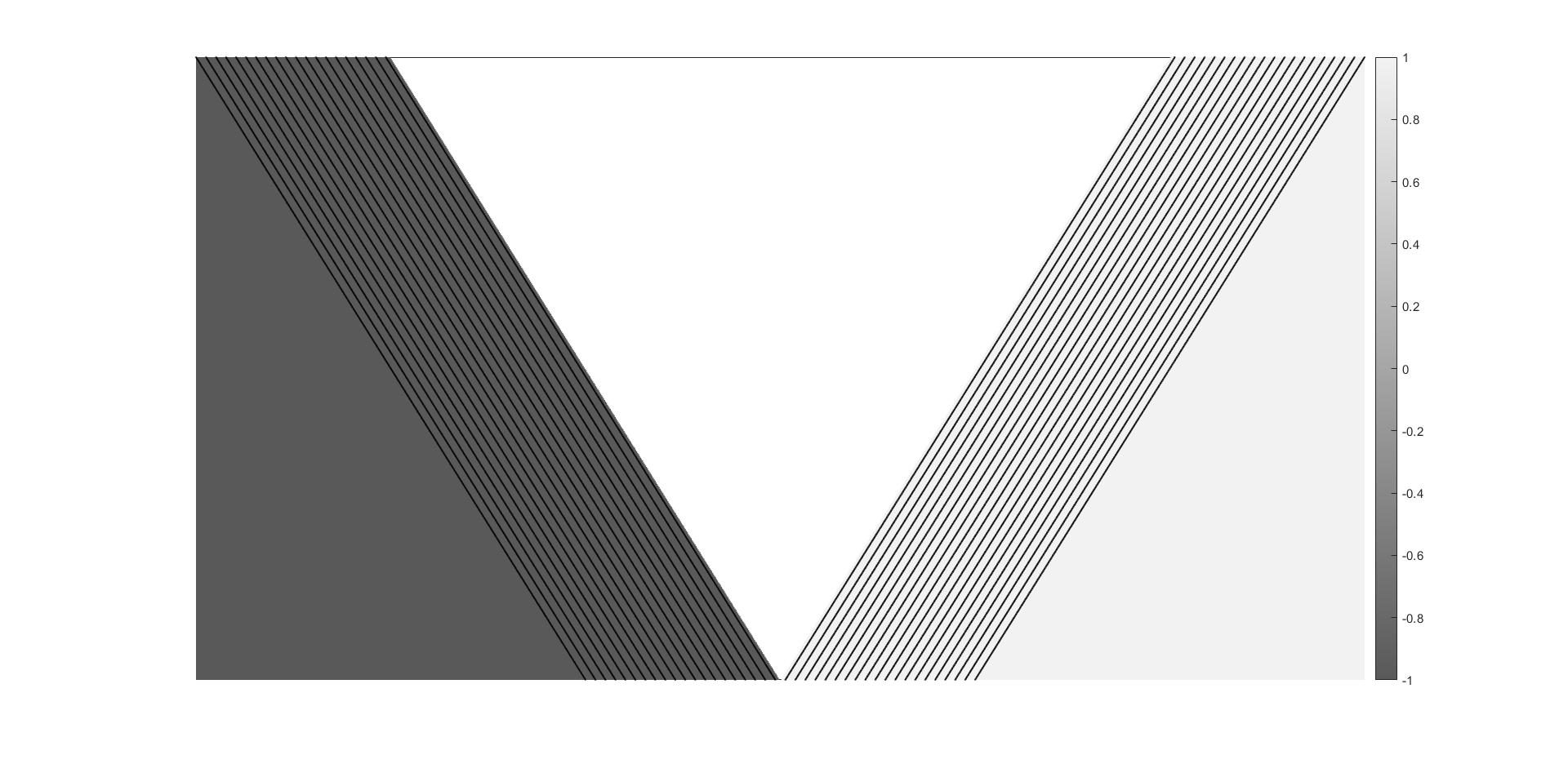}};
% Adding some axes
\draw[black,->](-4,-1.95) -- (4.095,-1.95);
\draw[black,->](-3.825,-2.2) -- (-3.825,2.25); 
% x-axis legend
\draw (-3.275,-2.35) node {\scriptsize $y=-(1+T)$};
\draw (3.5,-2.35) node {\scriptsize $y=1+T$};
% y-axis legend
\draw (-4.15,-1.85) node {\scriptsize \rotatebox{90}{$t = 0$}};
\draw (-4.15,2.15) node {\scriptsize \rotatebox{90}{$t = T$}};
% Labelling some agents
\draw (-1.35,-2.2) node {\scriptsize $-1 = x_1^0$};
\draw (-0.05,-2.2) node {\scriptsize $\ldots  \ldots \ldots$};
\draw (1.2,-2.2) node {\scriptsize $x_N^0 = 1$};
\end{scope}
%
%%%%%%%%%%%%%%%%%%%%%%%%%%%%%%%%%%%%%%%%%%%
% 		  THIRD PICTURE AND LEGEND        %
%%%%%%%%%%%%%%%%%%%%%%%%%%%%%%%%%%%%%%%%%%%
\begin{scope}[shift = {(0,-11)}]
% Incorporating the picture
\node(image) at (0,0)
     {\includegraphics[width = 0.56\textwidth]{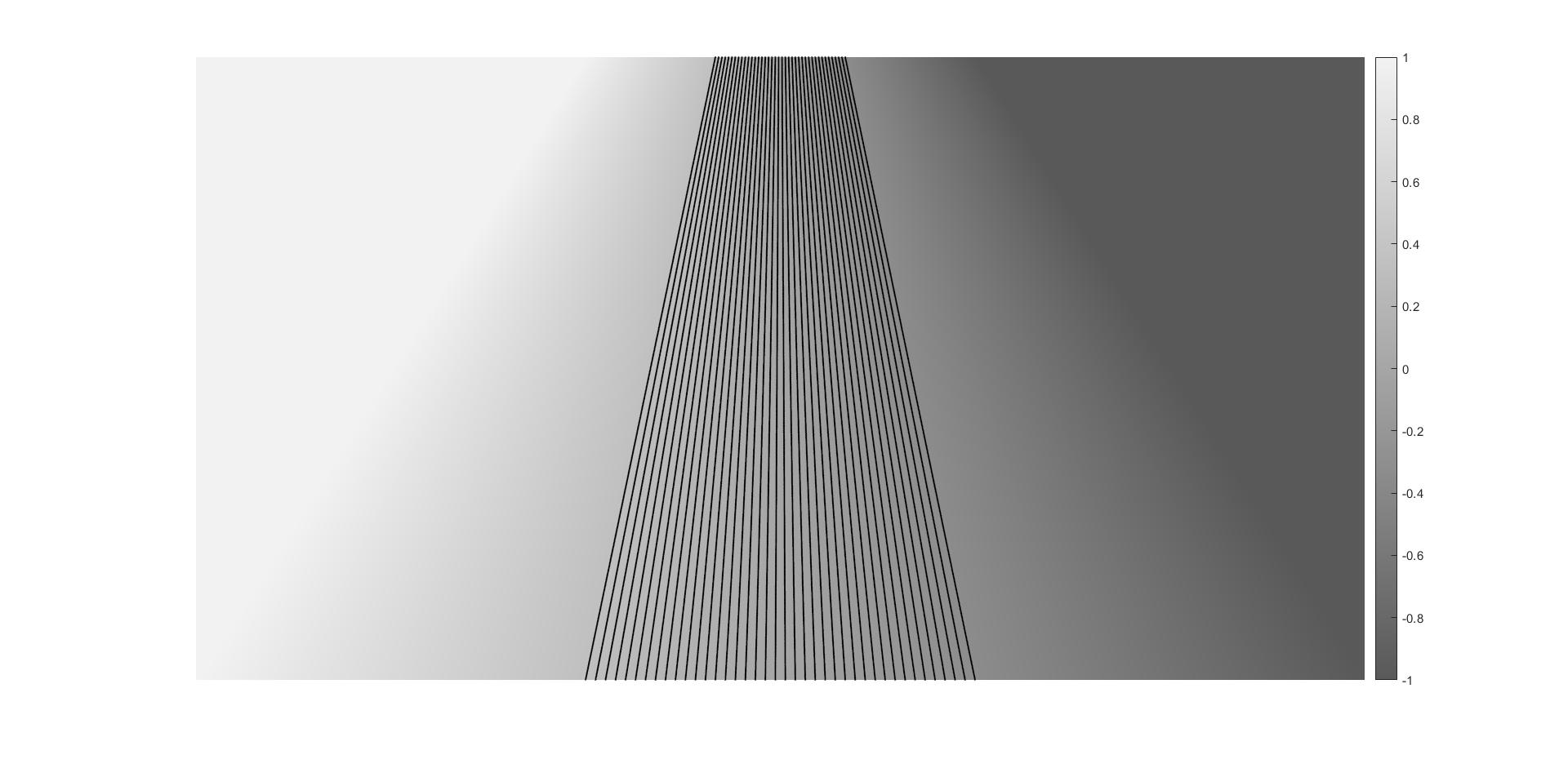}};
% Adding some axes
\draw[black,->](-4,-1.95) -- (4.095,-1.95);
\draw[black,->](-3.825,-2.2) -- (-3.825,2.25); 
% x-axis legend
\draw (-3.275,-2.35) node {\scriptsize $y=-(1+T)$};
\draw (3.5,-2.35) node {\scriptsize $y=1+T$};
% y-axis legend
\draw (-4.15,-1.85) node {\scriptsize \rotatebox{90}{$t = 0$}};
\draw (-4.15,2.15) node {\scriptsize \rotatebox{90}{$t = T$}};
% Labelling some agents
\draw (-1.35,-2.2) node {\scriptsize $-1 = x_1^0$};
\draw (-0.05,-2.2) node {\scriptsize $\ldots  \ldots \ldots$};
\draw (1.2,-2.2) node {\scriptsize $x_N^0 = 1$};
\end{scope}
\end{tikzpicture}
\caption{{\small Trajectories of $N= 40$ agents (foreground) with the magnitude of $u^*(t,y)$ (background) in the case where $\lambda > T$ (top), $\lambda \in (0,T]$ (middle) and $\lambda < 0$ (bottom).}}
\label{fig2}
\end{figure}
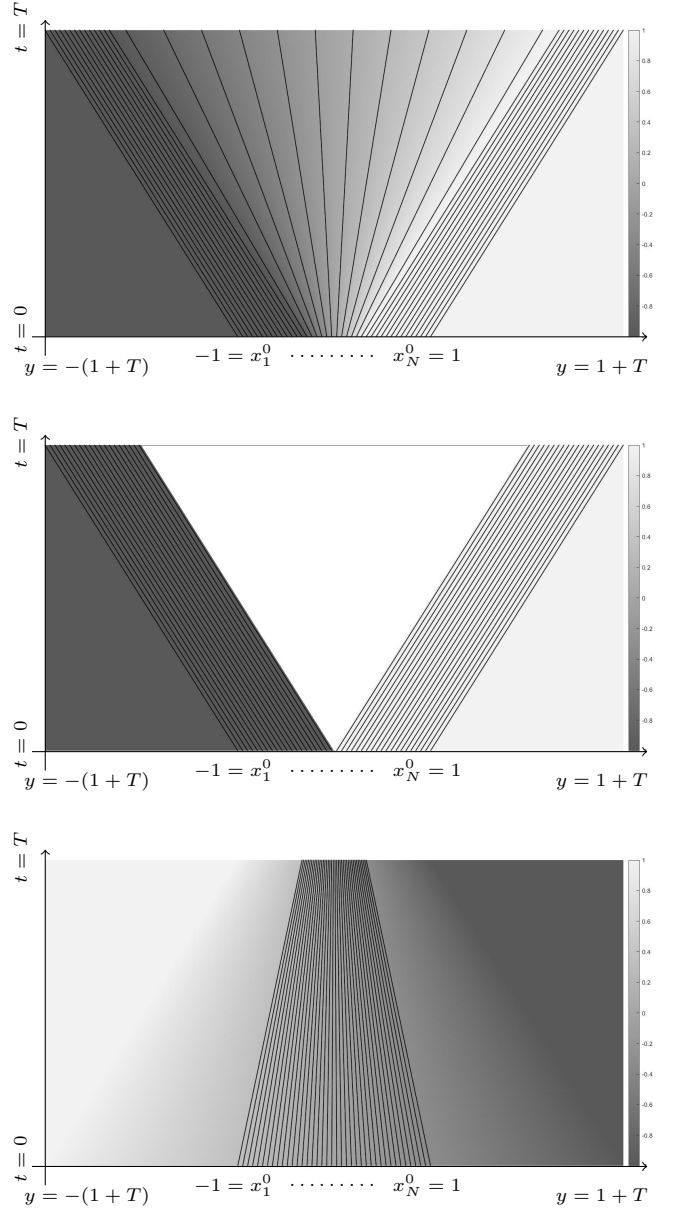

The optimal agent trajectories $\xb^*_N(\cdot)$ along with the velocity field $u^*(\cdot,\cdot)$ are illustrated for each scenario in Figure \ref{fig2}. In the sequel, we will need the following useful lemma.

\begin{lem} 
\label{l-cost} 
Let $u : [0,T] \times \R \rightarrow \R$ be measurable in time, Lipschitz in space and bounded. For each initial data $\xb^0_N$ of $\Pn$, define the control $\ub_N(\cdot)$ via $u_i(t):=u(t,x_i(t))$. If $\mu^0 \in \Pcal_c(\R^d)$ and the empirical measures $\mu_N^0:=\frac1N\sum_{i=1}^N \delta_{x_i^0}$ have uniformly compact support, then 
\begin{equation*}
\mu_N^0 \underset{N \rightarrow +\infty}{\weak} \mu^0 \quad \text{implies} \quad \Cpazo_N(\xb^0_N,\ub_N) \underset{N \rightarrow +\infty}{\longrightarrow} \Cpazo_\infty(\mu^0,u).
\end{equation*} 
\end{lem}

\begin{proof} 
Consider the trajectory $\mu_N(t):=\frac1N\sum_{i=1}^N \delta_{x_i(t)}$, where $\xb_N(\cdot)$ is the unique solution of \eqref{e-ODE}. By definition of the controls $\ub_N(\cdot)$, we know that $\mu_N(\cdot)$ solves 
\begin{equation*}
\left\{
\begin{aligned}
& \partial_t \mu(t) + \nabla \cdot (u(t,\cdot) \mu(t)) = 0, \\
& \mu(0) = \mu^0_N, 
\end{aligned}
\right.
\end{equation*}
and the regularity of $u(t,\cdot)$ ensures that such a solution is unique. By a direct computation, one can check that 
\begin{equation}
\label{e-cost-ug}
\Cpazo_N(\xb^0_N,\ub_N) = \Cpazo_\infty(\mu_N^0,u).
\end{equation}
Denote now by $\mu(\cdot)$ the unique solution of \eqref{e-PDE} with control $(t,x) \mapsto u(t,x)$ and initial datum $\mu_0$. By the Lipschitz regularity of $u(t,\cdot)$ and Proposition \ref{p-gronwall}, the distance estimate
\begin{equation}\label{e-g2}
W_1(\mu_N(t),\mu (t)) \leqslant e^{LT} W_1(\mu_N(0),\mu(0)), 
\end{equation}
holds for all times $t \in [0,T]$. Since the measures $(\mu_N(0))$ have uniformly bounded support and $u(\cdot,\cdot)$ is bounded, then $(\mu_N(t))$ and $\mu(t)$ have uniformly bounded support. Recall that $\Var(\mu)=\INTDom{x^2}{\R}{\mu(x)} - (\INTDom{x}{\R}{\mu(x)})^2$, and observe that the function $x \mapsto x^2$ can be replaced by a map $\phi(\cdot)$ with Lipschitz constant $L_2$, that coincides with the latter on the supports of $(\mu_N(t)),\mu(t)$. Notice now that the function $x \in \R \mapsto \phi(x)/L_2$ is 1-Lipschitz, hence
\begin{equation*}
\mbox{$\left| \INTDom{x^2}{\R}{(\mu_N(T)-\mu(T))(x)} \right| \leq L_2 W_1(\mu_N(T),\mu(T)).$}
\end{equation*}
We can also estimate $(\int_{\R}  x\, \textnormal{d}\mu(x))^2$ with similar computation, since the integrand is $1$-Lipschitz. Because $u(t,\cdot)$ is Lipschitz and bounded, the map $u^2(t,\cdot)$ is Lipschitz with some constant $L_u$. Merging these facts, one can check that 
\begin{equation*}
\begin{aligned}
\Big| \Cpazo_{\infty} (& \mu_N^0,u) - \Cpazo_\infty(\mu^0,u) \Big| \leq \INTSeg{W_1(\mu_N(t),\mu(t))}{t}{0}{T} \\
& + \frac{1}{|\lambda|} \Big( L_2 W_1(\mu_N(T),\mu(T)) + C W_1(\mu_N(T),\mu(T)) \Big),
\end{aligned}
\end{equation*}
where $C > 0$ is a constant uniform in $N \geq 1$. The proof of our claim follows by combining \eqref{e-cost-ug} and \eqref{e-g2}. 
\end{proof}

%We are now ready to prove our main result Theorem \ref{t-main}.

{\it Proof of Theorem \ref{t-main}.} We study the three cases separately.

We first consider the situation in which $\lambda>T$, and prove that the Lipschitz-in-space control $u^*(\cdot,\cdot)$ given by \eqref{e-f1} is optimal for $\Pin$. By contradiction, assume that there exists another Lipschitz control $\tilde{u}(\cdot,\cdot)$ such that $\Cpazo_\infty(\mu^0,\tilde{u}) < \Cpazo_\infty(\mu^0,u^*)$, and define $\tilde{\ub}_N(\cdot)$ as $\tilde{u}_i(t) := \tilde{u}(t,\tilde{x}_i(t))$. Then by Lemma \ref{l-cost}, there exists $N \geq 1$ large enough such that 
\begin{equation*}
\Cpazo_N(\xb^0_N,\tilde{\ub}_N) = \Cpazo_{\infty}(\mu^0_N,\tilde u) < \Cpazo_{\infty}(\mu^0_N,u^*) = \Cpazo_N(\xb^0_N,\ub^*_N). 
\end{equation*}
which contradicts the fact that $u^*(\cdot,\cdot)$ is optimal for $\Pn$, as proven in Section \ref{s-Pn}. The case $\lambda<0$ is analogous.

We now sketch the proof for the case $\lambda\in(0,T]$. By contradiction, assume that $\tilde{u}(\cdot,\cdot)$ is an optimal control for $\Pin$ such that $\tilde{u}(t,\cdot)$ is Lipschitz with constant $L > 0$. By Lemma \ref{l-cost}, for each $\eps>0$ there exists $N_{\eps} \geq 1$ such that 
\begin{equation}
\label{e-eps}
|\Cpazo_N(\xb(0),\tilde{\ub}_N) - \Cpazo_{\infty}(\mu_0,\tilde{u})|<\eps,
\end{equation}
where $\tilde{\ub}_N(\cdot)$ is defined as before by $\tilde{u}_i(t):=\tilde{u}(t,x_i(t))$ for almost every $t \in [0,T]$. Our goal now is to estimate
\begin{equation*}
c_N := \Cpazo_N(\xb^0_N,\tilde{\ub}_N) - \Cpazo_N(\xb^0_N,\ub_N^*),
\end{equation*} 
where $\ub^*_N$ is given by \eqref{e-c2}. Since $\ub^*_N(\cdot)$ is optimal for $\Pn$, it necessarily holds that $c_N\geq 0$. Then, we have two cases.
\begin{itemize}
\item If $c_N \geq d$ for some $d>0$ and all $N \geq 1$, there exists a Lipschitz approximation $\hat u(t,\cdot)$ of \eqref{e-f2} such that
\begin{equation*}
\Cpazo_N(\xb_N^0,\tilde{\ub}_N) - \Cpazo_N(\xb^0_N,\hat{\ub}_N) > d/2, 
\end{equation*}
independently of $N \geq 1$. By letting $N\to\infty$ in \eqref{e-eps}, and thus $\eps\to 0$, it then holds
\begin{equation*}
\Cpazo_{\infty}(\mu_0,\hat{u})\leq \Cpazo_{\infty}(\mu_0,\tilde{u})-d/2, 
\end{equation*}
which contradicts the optimality of $\tilde u(\cdot,\cdot)$.
\item If $(c_N)$ is not bounded from below by a positive constant, there exists a subsequence (that we do not relabel) such that $c_N\to 0$, i.e. the costs $\Cpazo_N(\xb^0_N,\tilde{\ub}_N)$ get arbitrarily close to the optimal value $\Cpazo_N(\xb_0^N,\ub_N^*)$. However, a direct computation of solutions of \eqref{e-ODE} with any $L$-Lipschitz control shows that the cost does not converge, which leads to a contradiction.
\end{itemize}

\section{Conclusions} \label{s-conclusion}

In this article, we showed that the problem of either maximizing or minimizing the variance for multi-agent optimal control problems and for their mean-field approximation as $N\to +\infty$ exhibit very different behaviors, depending on the relative weight between the quadratic control penalization and the variance. When minimizing the variance functional, or maximizing it with a sufficiently small time horizon, the optimal controls for $\Pn$ allow to build an optimal control for $\Pin$ that is Lipschitz-in-space. On the contrary, when the time horizon is large in the variance maximization problem, the discrete optimal controls cannot be extended into a regular vector field, which allows to prove that there exist no Lipschitz solutions to $\Pin$. 

This interplay between, on the one hand, the existence of Lipschitz solutions for mean-field optimal control problems, and on the other hand the construction of suitable Lipschitz feedbacks at the microscopic level has been investigated in greater generality in \cite{lipreg}.

{\small 
\bibliography{ControlWassersteinBib2}
}
\end{document}